\documentclass[12pt]{article}
\usepackage{latexsym, amssymb, amsthm}
\usepackage{dsfont}
\usepackage{color}

\DeclareMathAlphabet\gothic{U}{euf}{m}{n}

\makeatletter
\def\eqnarray{\stepcounter{equation}\let\@currentlabel=\theequation
\global\@eqnswtrue
\tabskip\@centering\let\\=\@eqncr
$$\halign to \displaywidth\bgroup\hfil\global\@eqcnt\z@
  $\displaystyle\tabskip\z@{##}$&\global\@eqcnt\@ne
  \hfil$\displaystyle{{}##{}}$\hfil
  &\global\@eqcnt\tw@ $\displaystyle{##}$\hfil
  \tabskip\@centering&\llap{##}\tabskip\z@\cr}

\def\endeqnarray{\@@eqncr\egroup
      \global\advance\c@equation\m@ne$$\global\@ignoretrue}

\def\@yeqncr{\@ifnextchar [{\@xeqncr}{\@xeqncr[5pt]}}
\makeatother

\begin{document}
\bibliographystyle{tom}

\newtheorem{lemma}{Lemma}[section]
\newtheorem{thm}[lemma]{Theorem}
\newtheorem{cor}[lemma]{Corollary}
\newtheorem{prop}[lemma]{Proposition}
\newtheorem{ddefinition}[lemma]{Definition}

\theoremstyle{definition}

\newtheorem{remark}[lemma]{Remark}
\newtheorem{exam}[lemma]{Example}

\newcommand{\gota}{\gothic{a}}
\newcommand{\gotb}{\gothic{b}}
\newcommand{\gotc}{\gothic{c}}
\newcommand{\gote}{\gothic{e}}
\newcommand{\gotf}{\gothic{f}}
\newcommand{\gotg}{\gothic{g}}
\newcommand{\gothh}{\gothic{h}}
\newcommand{\gotk}{\gothic{k}}
\newcommand{\gotl}{\gothic{l}}
\newcommand{\gotm}{\gothic{m}}
\newcommand{\gotn}{\gothic{n}}
\newcommand{\gotp}{\gothic{p}}
\newcommand{\gotq}{\gothic{q}}
\newcommand{\gotr}{\gothic{r}}
\newcommand{\gots}{\gothic{s}}
\newcommand{\gott}{\gothic{t}}
\newcommand{\gotu}{\gothic{u}}
\newcommand{\gotv}{\gothic{v}}
\newcommand{\gotw}{\gothic{w}}
\newcommand{\gotz}{\gothic{z}}
\newcommand{\gotA}{\gothic{A}}
\newcommand{\gotB}{\gothic{B}}
\newcommand{\gotG}{\gothic{G}}
\newcommand{\gotL}{\gothic{L}}
\newcommand{\gotS}{\gothic{S}}
\newcommand{\gotT}{\gothic{T}}

\newcounter{teller}
\renewcommand{\theteller}{(\alph{teller})}
\newenvironment{tabel}{\begin{list}%
{\rm  (\alph{teller})\hfill}{\usecounter{teller} \leftmargin=1.1cm
\labelwidth=1.1cm \labelsep=0cm \parsep=0cm}
                      }{\end{list}}

\newcounter{tellerr}
\renewcommand{\thetellerr}{(\roman{tellerr})}
\newenvironment{tabeleq}{\begin{list}%
{\rm  (\roman{tellerr})\hfill}{\usecounter{tellerr} \leftmargin=1.1cm
\labelwidth=1.1cm \labelsep=0cm \parsep=0cm}
                         }{\end{list}}

\newcounter{tellerrr}
\renewcommand{\thetellerrr}{(\Roman{tellerrr})}
\newenvironment{tabelR}{\begin{list}%
{\rm  (\Roman{tellerrr})\hfill}{\usecounter{tellerrr} \leftmargin=1.1cm
\labelwidth=1.1cm \labelsep=0cm \parsep=0cm}
                         }{\end{list}}

\newcounter{proofstep}
\newcommand{\nextstep}{\refstepcounter{proofstep}\vertspace \par 
          \noindent{\bf Step \theproofstep.} \hspace{5pt}}
\newcommand{\firststep}{\setcounter{proofstep}{0}\nextstep}

\newcommand{\Ni}{\mathds{N}}
\newcommand{\Qi}{\mathds{Q}}
\newcommand{\Ri}{\mathds{R}}
\newcommand{\Ci}{\mathds{C}}
\newcommand{\Ti}{\mathds{T}}
\newcommand{\Zi}{\mathds{Z}}
\newcommand{\Fi}{\mathds{F}}

\renewcommand{\proofname}{{\bf Proof}}

\newcommand{\vertspace}{\vskip10.0pt plus 4.0pt minus 6.0pt}

\newcommand{\simh}{{\stackrel{{\rm cap}}{\sim}}}
\newcommand{\ad}{{\mathop{\rm ad}}}
\newcommand{\Ad}{{\mathop{\rm Ad}}}
\newcommand{\alg}{{\mathop{\rm alg}}}
\newcommand{\clalg}{{\mathop{\overline{\rm alg}}}}
\newcommand{\Aut}{\mathop{\rm Aut}}
\newcommand{\arccot}{\mathop{\rm arccot}}
\newcommand{\capp}{{\mathop{\rm cap}}}
\newcommand{\rcapp}{{\mathop{\rm rcap}}}
\newcommand{\diam}{\mathop{\rm diam}}
\newcommand{\divv}{\mathop{\rm div}}
\newcommand{\dom}{\mathop{\rm dom}}
\newcommand{\codim}{\mathop{\rm codim}}
\newcommand{\RRe}{\mathop{\rm Re}}
\newcommand{\IIm}{\mathop{\rm Im}}
\newcommand{\tr}{{\mathop{\rm tr \,}}}
\newcommand{\Tr}{{\mathop{\rm Tr \,}}}
\newcommand{\Vol}{{\mathop{\rm Vol}}}
\newcommand{\card}{{\mathop{\rm card}}}
\newcommand{\rank}{\mathop{\rm rank}}
\newcommand{\supp}{\mathop{\rm supp}}
\newcommand{\sgn}{\mathop{\rm sgn}}
\newcommand{\essinf}{\mathop{\rm ess\,inf}}
\newcommand{\esssup}{\mathop{\rm ess\,sup}}
\newcommand{\Int}{\mathop{\rm Int}}
\newcommand{\lcm}{\mathop{\rm lcm}}
\newcommand{\loc}{{\rm loc}}
\newcommand{\HS}{{\rm HS}}
\newcommand{\Trn}{{\rm Tr}}
\newcommand{\n}{{\rm N}}
\newcommand{\WOT}{{\rm WOT}}

\newcommand{\at}{@}

\newcommand{\mod}{\mathop{\rm mod}}
\newcommand{\spann}{\mathop{\rm span}}
\newcommand{\one}{\mathds{1}}

\hyphenation{groups}
\hyphenation{unitary}

\newcommand{\tfrac}[2]{{\textstyle \frac{#1}{#2}}}

\newcommand{\ca}{{\cal A}}
\newcommand{\cb}{{\cal B}}
\newcommand{\cc}{{\cal C}}
\newcommand{\cd}{{\cal D}}
\newcommand{\ce}{{\cal E}}
\newcommand{\cf}{{\cal F}}
\newcommand{\ch}{{\cal H}}
\newcommand{\chs}{{\cal HS}}
\newcommand{\ci}{{\cal I}}
\newcommand{\ck}{{\cal K}}
\newcommand{\cl}{{\cal L}}
\newcommand{\cm}{{\cal M}}
\newcommand{\cn}{{\cal N}}
\newcommand{\co}{{\cal O}}
\newcommand{\cp}{{\cal P}}
\newcommand{\cs}{{\cal S}}
\newcommand{\ct}{{\cal T}}
\newcommand{\cx}{{\cal X}}
\newcommand{\cy}{{\cal Y}}
\newcommand{\cz}{{\cal Z}}

\newlength{\hightcharacter}
\newlength{\widthcharacter}
\newcommand{\covsup}[1]{\settowidth{\widthcharacter}{$#1$}\addtolength{\widthcharacter}{-0.15em}\settoheight{\hightcharacter}{$#1$}\addtolength{\hightcharacter}{0.1ex}#1\raisebox{\hightcharacter}[0pt][0pt]{\makebox[0pt]{\hspace{-\widthcharacter}$\scriptstyle\circ$}}}
\newcommand{\cov}[1]{\settowidth{\widthcharacter}{$#1$}\addtolength{\widthcharacter}{-0.15em}\settoheight{\hightcharacter}{$#1$}\addtolength{\hightcharacter}{0.1ex}#1\raisebox{\hightcharacter}{\makebox[0pt]{\hspace{-\widthcharacter}$\scriptstyle\circ$}}}
\newcommand{\scov}[1]{\settowidth{\widthcharacter}{$#1$}\addtolength{\widthcharacter}{-0.15em}\settoheight{\hightcharacter}{$#1$}\addtolength{\hightcharacter}{0.1ex}#1\raisebox{0.7\hightcharacter}{\makebox[0pt]{\hspace{-\widthcharacter}$\scriptstyle\circ$}}}

\thispagestyle{empty}

\vspace*{1cm}
\begin{center}
{\Large\bf Analyticity of the Dirichlet-to-Neumann semigroup \\[2mm]
on continuous functions 
\\[10mm]
\large A.F.M. ter Elst$^1$ and E.M. Ouhabaz$^2$}

\end{center}

\vspace{5mm}

\begin{center}
{\bf Abstract}
\end{center}

\begin{list}{}{\leftmargin=1.8cm \rightmargin=1.8cm \listparindent=10mm 
   \parsep=0pt}
\item
Let $\Omega$ be a bounded open subset with $C^{1+\kappa}$-boundary for some $\kappa > 0$.
Consider the Dirichlet-to-Neumann operator associated to the 
elliptic operator $- \sum \partial_l ( c_{kl} \, \partial_k ) + V$,
where the $c_{kl} = c_{lk}$ are  H\"older continuous 
and $V \in L_\infty(\Omega)$ are real valued.
We prove that the  Dirichlet-to-Neumann operator generates a $C_0$-semigroup
on  the space $C(\partial \Omega)$ 
which is in addition holomorphic with angle $\frac{\pi}{2}$. 
We also show that the kernel of the semigroup has Poisson bounds 
on the complex right half-plane.
As a consequence we obtain an optimal holomorphic functional calculus
and maximal regularity on $L_p(\Gamma)$ for all $p \in (1,\infty)$.
\end{list}

\vspace{5mm}
\noindent
July 2017

\vspace{5mm}

\vspace{5mm}

\vspace{15mm}

\noindent
{\bf Home institutions:}    \\[3mm]
\begin{tabular}{@{}cl@{\hspace{10mm}}cl}
1. & Department of Mathematics  & 
  2. & Institut de Math\'ematiques de Bordeaux \\
& University of Auckland   & 
  & Universit\'e de Bordeaux, UMR 5251,  \\
& Private bag 92019 & 
  &351, Cours de la Lib\'eration  \\
& Auckland 1142 & 
  &  33405 Talence \\
& New Zealand  & 
  & France\\
  & terelst@math.auckland.ac.nz&
 & Elmaati.Ouhabaz@math.u-bordeaux.fr \\[8mm]
\end{tabular}

\newpage

\section{Introduction} \label{Sdtnpcont1}

The Dirichlet-to-Neumann operator $\cn$ is a self-adjoint 
operator in $L_2(\Gamma)$, where $\Gamma$ is the boundary of a
bounded set $\Omega \subset \Ri^d$ with Lipschitz boundary.
It is defined with respect to a symmetric differential operator on $\Omega$.
The Dirichlet-to-Neumann operator occurs in inverse problems 
(the Calder\'on problem) and attracted recently a lot of 
analytic interest \cite{Esc1},  \cite{Eng}, \cite{ArM2},
 \cite{BeE1},  \cite{AE3}, 
\cite{GG}, \cite{EO4}, \cite{EO6} and the references therein. 
Under some conditions on the boundary regularity of $\Omega$ 
and on the coefficients of the symmetric differential operator on $\Omega$
it follows that the semigroup $S$ generated by $-\cn$ extends consistently to a 
$C_0$-semigroup on $L_p$ for all $p \in [1,\infty)$.
In this paper we show that if $\Omega$ is a $C^{1+\kappa}$-domain 
and if the coefficients of the symmetric differential operator on $\Omega$
are H\"older continuous, then the semigroup $S$ leaves $C(\Gamma)$ 
invariant and the restriction of $S$ to $C(\Gamma)$ is a $C_0$-semigroup
on $C(\Gamma)$.  We prove in addition that this semigroup  is holomorphic with the optimal angle $\frac{\pi}{2}$.

For the precise statement of the theorem we need to introduce some notation.
Let $\Omega \subset \Ri^d$ be a bounded open set with Lipschitz boundary.
Write $\Gamma = \partial \Omega$.
Further, for all $k,l \in \{ 1,\ldots,d \} $ let $c_{kl} \colon \Omega \to \Ri$ 
be a bounded measurable function such that $c_{kl} = c_{lk}$
for all $k,l \in \{ 1,\ldots,d \} $.
We assume ellipticity, that is there exists a $\mu > 0$ such that 
\[
\sum_{k,l=1}^d c_{kl}(x) \, \xi_k \, \overline{\xi_l}
\geq \mu \, |\xi|^2
\]
for all $\xi \in \Ci^d$ and $x \in \Omega$.
Let $A_D = - \sum_{k,l=1}^d \partial_l ( c_{kl} \, \partial_k )$
subject to Dirichlet boundary conditions.
Let $V \colon \Omega \to \Ri$ be a bounded measurable function.
We always assume that the operator $A_D + V$ is invertible.

If $u \in W^{1,2}(\Omega)$ and $\psi \in L_2(\Gamma)$, 
then we say that $u$ has {\bf weak conormal derivative} $\psi$ if 
\[
\sum_{k,l=1}^d \int_\Omega c_{kl} \, (\partial_k u) \, \overline{\partial_l v}
   + \int_\Omega V \, u \, \overline v 
= \int_\Gamma \psi \, \overline{\Tr v}
\]
for all $v \in W^{1,2}(\Omega)$.
In that case we write $\partial_\nu u = \psi$. Formally, 
\[
\partial_\nu u = \sum_{k,l=1}^d n_l \, c_{kl} \, \partial_l u, 
\]
where $(n_1, \ldots, n_d)$ is the outer normal vector to $\Omega$. 

We define briefly the Dirichlet-to-Neumann operator $\cn_V$.
Let $\varphi \in H^{1/2}(\Gamma)$.
Since the operator $A_D + V$ is invertible we can uniquely solve the Dirichlet problem
\begin{eqnarray*}
-\sum_{k,l= 1}^d \partial_l \Big( c_{kl} \partial_k\,  u \Big) + V \, u & = & 0 
      \quad \mbox{weakly on } \Omega,   \\[0pt]
\Tr u & = & \varphi \nonumber
\end{eqnarray*}
with $u \in W^{1,2}(\Omega)$.
If $u$ has a weak conormal derivative, then $\varphi \in D(\cn_V)$ and 
$\cn_V \varphi = \partial_\nu u$.
The operator $\cn_V$ is a lower-bounded self-adjoint operator on $L_2(\Gamma)$.
For more details see \cite{EO4} Section~2  or \cite{EO6} Section~2.
Let $S^V$ be the $C_0$-semigroup on $L_2(\Gamma)$ generated by 
$- \cn_V$.
Obviously the semigroup $S^V$ is holomorphic on $L_2(\Gamma)$ with 
angle~$\frac{\pi}{2}$.
Our contribution in this paper is threefold.

\begin{enumerate}
\item
We prove that $S^V$ extends consistently to a holomorphic $C_0$-semigroup
on $L_p(\Gamma)$ with angle $\frac{\pi}{2}$ for all $p \in [1,\infty)$.

\item
If $p \in (1,\infty)$ and $\lambda > 0$ is large enough, then $\cn_V + \lambda \, I$
has a bounded $H^\infty(\Sigma(\theta))$-functional calculus on $L_p(\Gamma)$
for all $\theta \in (0,\frac{\pi}{2})$.
Moreover, $\cn_V$ has maximal $L_r$-regularity on $L_p(\Gamma)$ for all $r \in (1,\infty)$.

\item
The central point of this paper is to prove existence and holomorphy of the 
semigroup on the space of continuous functions $C(\Gamma)$.
\end{enumerate}

The following is our main theorem.

\begin{thm} \label{tdtnpcont101}
Adopt the above assumptions and definitions.
Suppose that there exists a $\kappa > 0$ such that 
$\Omega$ has a $C^{1 + \kappa}$-boundary.
Further suppose that the $c_{kl}$ are H\"older continuous.
Then $S^V$ leaves the space $C(\Gamma)$ invariant.
Moreover, $(S^V_t|_{C(\Gamma)})_{t > 0}$ is a $C_0$-semigroup
on $C(\Gamma)$ which is holomorphic with angle $\frac{\pi}{2}$.
\end{thm}

This theorem has been proved by Engel \cite{Eng}, Theorem~2.1
for the classical Dirichlet-to-Neumann operator, that is if $c_{kl} = \delta_{kl}$
and $V = 0$ and if $\Omega$ has a $C^\infty$-boundary. 
In this case, $\cn$ is the 
perturbation of the square root of the Laplace--Beltrami operator on 
$\partial \Omega$ 
by a pseudo-differential operator of order zero. 
Hence generation and analyticity of the semigroup can be treated by 
perturbation arguments.  

In case $\Omega$ has a $C^\infty$-boundary and both the coefficients $c_{kl}$ 
and potential $V$ are $C^\infty$-functions, then Escher \cite{Esc1} 
states that $(S^V_t|_{C(\Gamma)})_{t > 0}$ is a $C_0$-semigroup
on $C(\Gamma)$ which is holomorphic on a sector with  some positive angle. 
Note, however, that  
a positivity condition is 
missing in \cite{Esc1}, see the discussion before Theorem~6.3 in \cite{DGK2}.

If $\Omega$ is a ball in $\Ri^2$, one considers the Laplacian
$c_{kl} = \delta_{kl}$ and $V$ is constant, with $0 \not\in \sigma(A_D + V)$,
then Daners--Gl\"uck--Kennedy \cite{DGK2}, Theorem~6.3, proved that 
$(S^V_t|_{C(\Gamma)})_{t > 0}$ is a $C_0$-semigroup
on $C(\Gamma)$.  No  holomorphy is proved there. 

If $\Omega$ has a $C^{1,1}$-boundary (in $\Ri^d$), 
one considers again the Laplacian $c_{kl} = \delta_{kl}$
and $V \in L_\infty(\Omega,\Ri)$, with $0 \not\in \sigma(A_D + V)$,
then Arendt--ter Elst \cite{AE8}, Theorem~5.3(b) showed that 
$(S^V_t|_{C(\Gamma)})_{t > 0}$ is a $C_0$-semigroup
on $C(\Gamma)$.  No  holomorphy is proved in that paper. 

If the $c_{kl}$ are Lipschitz continuous, $V \geq 0$ and $\Omega$ 
merely has a Lipschitz boundary, 
then Arendt--ter Elst \cite{AE8}, Theorem~5.3(a) showed that 
$S^V$ leaves the space $C(\Gamma)$ invariant and that 
$(S^V_t|_{C(\Gamma)})_{t > 0}$ is a $C_0$-semigroup
on $C(\Gamma)$.
Actually, $V$ is allowed to be slightly negative  and 
the condition in Theorem~5.3(c) of that paper is that the operator 
$A_D + V$ is positive.
There is no holomorphy obtained for the semigroup in \cite{AE8}.

In this paper the condition on the boundary is stronger than in \cite{AE8}, Theorem~5.3(c),
since we require a $C^{1+\kappa}$-boundary instead of a Lipschitz boundary.
On the other hand, the condition of the principal coefficients is weaker, 
that is H\"older continuity instead of Lipschitz continuity. 
Moreover, $V$ is allowed to be (very) negative, we only require
that $V$ is real measurable bounded and $0 \not\in \sigma(A_D + V)$.
In addition we prove analyticity of the semigroup with optimal 
angle equal to~$\frac{\pi}{2}$.

Under the conditions of Theorem~1.1 the semigroup $S^V$ has a kernel satisfying Poisson
bounds for positive  time, see \cite{EO6}, Theorem~1.1. We rely  heavily on this result. 
In Section~\ref{Sdtnpcont2} we state the Poisson bounds and use them to 
prove that the restriction of $S^V$ to $C(\Gamma)$ is a $C_0$-semigroup. The fact that 
$S^V$ leaves $C(\Gamma)$ invariant follows from elliptic regularity. 
In Section~\ref{Sdtnpcont3} we use iteration to deduce that 
the Poisson bounds extend to any sector in the complex with angle 
strictly less than $\frac{\pi}{2}$.
Then the analyticity on the open right half-plane follows. 
Our approach gives that the semigroup
$S^V$ extends consistently to a holomorphic semigroup on $L_p(\Gamma)$ with angle 
$\frac{\pi}{2}$ for all $p \in [1, \infty)$. 
We also take advantage of the Poisson bounds in the complex half-plane to obtain
a holomorphic functional calculus, bounded imaginary powers and maximal regularity on 
$L_p(\Gamma)$ for all $p \in (1,\infty)$.

\section{$C_0$-semigroup on $C(\Gamma)$} \label{Sdtnpcont2}

In this section we prove a part of the main theorem.
It concerns the fact that the operator $-\cn_V$ generates a strongly continuous semigroup on $C(\Gamma)$. 
One of the main ingredients in the proof is the following Poisson bound proved 
recently in \cite{EO6}. 
The assumptions here are the same as in 
Theorem \ref{tdtnpcont101}. 

\begin{thm} \label{tdtnpcont201} 
Suppose $\Omega \subset \Ri^d$ is bounded connected 
with a $C^{1+ \kappa}$-boundary 
$\Gamma$ for some $\kappa \in (0,1)$.
Suppose also each $c_{kl} = c_{lk}$ is real valued and 
H\"older continuous on~$\Omega$.
Let $V \in L_\infty(\Omega,\Ri)$ and suppose that 
$0 \notin\sigma(A_D +V)$.
Denote by $\cn_V$ the corresponding Dirichlet-to-Neumann operator. 
Then the semigroup $S^V = (S^V_t)_{t \ge 0}$ generated by $-\cn_V$ on $L_2(\Gamma)$ 
has a kernel $K^V$ and  there exists a $c > 0$ such that 
\[
| K^V_t(w_1,w_2) | 
\leq \frac{c \, (t \wedge 1)^{-(d-1)} \, e^{-\lambda_1 t}}
         {\displaystyle \Big( 1 + \frac{|w_1-w_2|}{t} \Big)^d }
\]
for all $w_1, w_2 \in \Gamma$ and $t > 0$, where $\lambda_1 $ is the 
first eigenvalue of the operator~$\cn_V$.
\end{thm}

It follows from the Poisson bounds that $S^V$ extends consistently to an
exponentially bounded semigroup on $L_p(\Gamma)$ for all $p \in [1,\infty]$,
which is a $C_0$-semigroup if $p \in [1,\infty)$.
With abuse of notation we denote this semigroup on $L_p(\Gamma)$ 
also by~$S^V$.

The first, and essential part for the proof of the semigroup on $C(\Gamma)$ 
is that the semigroup $S^V$ leaves $C(\Gamma)$ invariant.

\begin{lemma} \label{ldtnpcont202} 
If $t > 0$, then $S^V_t(L_\infty(\Gamma)) \subset C(\Gamma)$. 
In particular, $S^V_t(C(\Gamma)) \subset C(\Gamma)$ for all $t > 0$.
\end{lemma}
\begin{proof} 
The operator $\cn_V$ is self-adjoint on $L_2(\Gamma)$ and has compact resolvent.
Therefore there exists 
an orthonormal basis $(\varphi_n)_{n \in \Ni}$ of $L_2(\Gamma)$ such that $\varphi_n$ is an 
eigenfunction of $\cn_V$ with corresponding eigenvalue $\lambda_n$ for all $n \in \Ni$.
Let $n \in \Ni$.
By the definition of $\cn$ there exists a  function 
$u_n \in W^{1,2}(\Omega)$, harmonic in $\Omega$ and such that 
$\Tr u_n = \varphi_n$.
Clearly
$\varphi_n \in L_\infty(\Gamma)$ and 
\[
\int_\Omega \sum_{k,l=1}^d c_{kl} (\partial_k u_n) \, \overline{\partial_l v} 
    + \int_\Omega V \, u_n \, \overline v 
= \lambda_n \int_{\Gamma} \varphi_n \, \overline{\Tr v}
\]
for every $v \in W^{1,2}(\Omega)$.
Then from Theorem 3.14 ii) in \cite{Nit4} it follows that $u_n$ is H\"older continuous on $\Omega$.
In particular,  $\varphi_n \in C(\Gamma)$ for every $n \in \Ni$. 

The heat kernel $K^V$ of $\cn$ can be written as
\begin{equation}\label{eq2.1}
K^V_t (w_1, w_2) = \sum_{n=1}^\infty e^{-\lambda_n t} \varphi_n(w_1) \, \overline{\varphi_n(w_2)}.
\end{equation}
For each $t > 0$, the sum converges uniformly in $C(\Gamma) \times C(\Gamma)$.
In order to see this we argue as in \cite{AE7}, Theorem~2.7.
Since 
$e^{-\lambda_n \frac{t}{3}} \varphi_n = S_{\frac{t}{3}} \varphi_n$ one obtain from the 
$L_2$--$L_\infty$ estimate of $S_t$ (or from the Poisson bound) that 
\[
\| \varphi_n \|_{L_\infty(\Gamma)} 
\le c \, e^{\delta t} \, t^{-(d-1)/2} \, e^{\lambda_n \, \frac{t}{3}}
\]
for some constants $c,\delta > 0$ independent of $n$ and $t$.
  Therefore
\[
\sum_{n=1}^\infty |e^{-\lambda_n t} \, \varphi_n(w_1) \, \overline{\varphi_n(w_2)}| 
\le c^2 \, e^{2\delta t} \, t^{-(d-1)} \sum_{n=1}^\infty e^{-\lambda_n \frac{t}{3}}
\]
for all $\omega_1,\omega_2 \in \Gamma$.
The sum on the RHS is convergent since it coincides with the 
trace of the operator $S_{\frac{t}{3}}$.
Hence the sum on the RHS of (\ref{eq2.1})  is uniformly convergent in $\Gamma \times \Gamma$.

It follows now that $K^V_t$ is continuous on $\Gamma \times \Gamma$.
Therefore $S^V_t \varphi \in C(\Gamma)$ for every $\varphi \in L_\infty(\Gamma)$. 
\end{proof}

We may define $ T^V_t $ to be the restriction of $S^V_t$ to $C(\Gamma)$ for all $t > 0$.
The family $T^V := (T^V_t)_{t>0}$ is clearly a semigroup on $C(\Gamma)$.

\begin{prop} \label{pdtnpcont203} 
The semigroup $T^V$ is strongly continuous  on $C(\Gamma)$. 
\end{prop}
\begin{proof} 
First, we observe that by the Poisson bound of Theorem~\ref{tdtnpcont201}
there exists a constant 
$M > 0$ such that 
\[
\| T^V_t \|_{C(\Gamma) \to C(\Gamma)} \le M
\]
for all $t \in (0, 1]$.
Hence in order to prove strong continuity it is enough to prove that
\begin{equation}\label{eq2.3}
\lim_{t \to 0} T^V_t (P_{| \Gamma}) = P_{| \Gamma}
\end{equation}
for every polynomial $P$, where the limit is taken in the $C(\Gamma)$-sense.
We proceed in two steps.

Assume first that $V = 0$ and set $T = T^0$.
Since $\one \in D(\cn)$ with $\cn (\one) = 0$ it follows that 
\[
\int_\Gamma K_t(w_1, w_2) \, dw_2 = 1
\]
for every $t > 0$ and $w_1 \in \Gamma$.
Therefore
\[
\Big( T_t(P_{| \Gamma}) - P_{| \Gamma} \Big) (w_1) 
= \int_\Gamma K_t(w_1, w_2) \, (P(w_2) - P(w_1)) \, dw_2
.  \]
Now we use the Poisson bound for $t  \in (0, 1]$ and obtain that there are 
constants $c, c', c'' > 0$ such that 
\begin{eqnarray*}
| (T_t(P_{| \Gamma}) - P_{| \Gamma})(w_1) | 
&\le& c \, t^{-(d-1)} \int_\Gamma \Big( 1 + \frac{|w_1-w_2|}{t} \Big)^{-d } |w_1-w_2|^{\frac{1}{2}} \, d w_2\\
&\le& c' \, t^{\frac{1}{2}} \, t^{-(d-1)} 
\int_\Gamma \Big( 1 + \frac{|w_1-w_2|}{t} \Big)^{-(d-\frac{1}{2}) } \, d w_2\\
&\le& c'' \, t^{\frac{1}{2}}
\end{eqnarray*}
uniformly for all $w_1 \in \Gamma$ and $t \in (0,1]$.
This shows (\ref{eq2.3}) in the $C(\Gamma)$-sense. 

For a general $V \in L_\infty(\Omega)$ we proceed by perturbation.
It follows from \cite{EO6}, Corollary~5.6 that 
\[
\cn_V = \cn + Q = \cn + \gamma_0^* \, M_V \, \gamma_V
,  \]
where $\gamma_V$ is the harmonic lifting associated with 
$-\sum_{k,l=1}^d \partial_k( c_{kl} \, \partial_l) + V$, 
$\gamma_0$ is the harmonic lifting associated with 
$-\sum_{k,l=1}^d \partial_k( c_{kl} \, \partial_l)$ and  $M_V$ is the multiplication operator by~$V$. 
By \cite{EO6}, Proposition~5.5(d) the operator $\gamma_V$ is bounded from $C(\Gamma)$ to 
$L_\infty(\Omega)$ and 
a combination of Lemma~5.4 and Propositions 4.3 and~5.3 shows that 
$\gamma_0^* $ is bounded from $ L_\infty(\Omega)$ to $C(\Gamma)$.
Therefore the operator $Q$ is bounded on $C(\Gamma)$. 
Since the part of $-\cn$ on $C(\Gamma)$ generates a $C_0$-semigroup,  the same result 
holds for the part of $-\cn_V$ by bounded perturbation.
\end{proof}

\section{Holomorphy and Poisson bounds} \label{Sdtnpcont3}

We have proved in the previous section that the semigroup $T^V$ generated by
the part of $-\cn_V$ on $C(\Gamma)$ is 
strongly continuous on $C(\Gamma)$.
Here we prove that this semigroup is holomorphic with angle~$\frac{\pi}{2}$.  

Let us denote by 
$\Sigma(\theta) = \{ z \in \Ci:  z \not= 0 \mbox{ and } |\arg(z)| < \theta \}$  the 
open sector of the right half-plane with angle $\theta \in (0, \frac{\pi}{2})$. 
We start with proving Poisson bounds.

\begin{lemma} \label{ldtnpcont306} 
Let $z_0 \in \Sigma(\frac{\pi}{2d})$ with $|z_0| = 1$.
Then the heat kernel of the semigroup $(S^V_{t z_0})_{t > 0}$ satisfies 
Poisson bounds.
\end{lemma}
\begin{proof}
First, let $\varepsilon \in (0,\frac{1}{d})$ be such that 
$|\arg z_0| < \varepsilon$.
By the Poisson bound of Theorem~\ref{tdtnpcont201} it follows from Proposition~3.3 in \cite{DR1} 
(see also Theorem~1 in \cite{YangZhu}) that there exist 
$c, \delta > 0$  such that 
the kernel $K^V_z$ of $S^V_z$ satisfies
\begin{equation}\label{eq4.1}
| K^V_z(w_1, w_2) | 
\le c \, (1 \wedge \RRe z)^{-(d-1)} \, e^{\delta \RRe z} 
    {\displaystyle \Big( 1 + \frac{|w_1-w_2|}{ |z|} \Big)^{d(1-\varepsilon) }}
\end{equation}         
for all $z \in \Sigma(\varepsilon \frac{\pi}{2})$ and $w_1,w_2 \in \Gamma$.
Note that $d(1- \varepsilon) > d-1$ and hence there exists a $c' > 0$ such that
\[
\max  \left( \sup_{w_2 \in \Gamma} \int_\Gamma | K^V_{t z_0}(w_1,w_2) | \, dw_1 , 
    \sup_{w_1 \in \Gamma} \int_\Gamma | K^V_{t z_0}(w_1,w_2) | \, dw_2 \right)
\le c' \, e^{t \delta \RRe z_0}
\]
for all $t > 0$. 
Therefore
\[
\|S^V_{t z_0}\|_{p \to p} 
\leq c' \, e^{t \delta \RRe z_0} 
\]
for all $p \in [1,\infty]$ and $t > 0$.
Moreover, (\ref{eq4.1}) implies that 
\[
\| T^V_{tz_0} \|_{1 \to \infty} 
\le c \, (\RRe z_0)^{-(d-1)} \, (1 \wedge t)^{-(d-1)} e^{t \delta \RRe z_0}
\]
for all $t > 0$.
Hence by interpolation, we obtain  $L_p$--$L_q$ estimates 
for all $p, q \in [1, \infty]$ with $p \le q$.

Write $\Lambda_0 := z_0 \, \cn_V$.
Then $- \Lambda_0$ is the generator of the semigroup $(S^V_{t z_0})_{t > 0}$.

Secondly, for any Lipschitz function $g$ on $\Gamma$, the commutator of $\Lambda_0$ 
with $M_g$ (the multiplication operator with $g$) satisfies
\[
{} [ \Lambda_0, M_g] = z_0 \, [\cn_V, M_g]
. \]
By  \cite{EO6}, Theorem~7.3 there exists a $c > 0$ such that 
\[
 \| [\cn_V, M_g] \|_{p \to p}  
\le c \, {\rm Lip}_\Gamma(g)
\]
for all $g \in C^{0,1}(\Gamma)$.
Here  
${\rm Lip}_\Gamma(g) 
    := \sup_{w_1,w_2 \in \Gamma, w_1 \not= w_2} \frac{ | g(w_1) - g(w_2) |}{| w_1-w_2|}$.
Therefore
\[
\| [ \Lambda_0, M_g] \|_{p \to p}  
\le c \, {\rm Lip}_\Gamma(g) 
\]
for all $g \in C^{0,1}(\Gamma)$. 

Thirdly, the Schwartz kernel of $\Lambda_0$ is $z_0 K_{\cn_V}$, where $K_{\cn_V}$ is the 
Schwartz kernel of $\cn_V$.
It is proved in \cite{EO6}, Proposition~6.5 that there is a $c > 0$ such that 
\[
| K_{\cn_V} (w_1, w_2) | \le \frac{c}{ | w_1 - w_2|^d}
\]
for all  $w_1, w_2 \in \Gamma$.
The same estimate is obviously satisfied by the Schwartz kernel of $\Lambda_0$. 

Combining these three observations we can repeat the arguments of \cite{EO6}, Section~8 
and obtain a Poisson bound for the heat kernel associated with the operator $\Lambda_0$.
\end{proof}

The alluded Poisson bounds on the complex right half-plane are as follows.

\begin{thm} \label{tdtnpcont307} 
Suppose $\Omega \subset \Ri^d$ is bounded connected with a 
$C^{1+ \kappa}$-boundary 
$\Gamma$ for some $\kappa \in (0,1)$.
Suppose also that each $c_{kl} = c_{lk}$ is real valued and 
H\"older continuous on~$\Omega$.
Let $V \in L_\infty(\Omega,\Ri)$ and suppose that 
$0 \notin\sigma(A_D +V)$.
Let $\theta \in (0, \frac{\pi}{2})$. 
Then  there exists a $c > 0$ such that 
\[
| K^V_z(w_1,w_2) | 
\leq \frac{c \, ( 1 \wedge \RRe z )^{-(d-1)} \, e^{-\lambda_1 \RRe z} }
         {\displaystyle \Big( 1 + \frac{|w_1-w_2|}{ |z|} \Big)^d }
\]
for all $z \in \Sigma(\theta)$ and $w_1, w_2 \in \Gamma$, where $\lambda_1 $ is the 
first eigenvalue of the operator $\cn_V$.
\end{thm}
\begin{proof}
Define the sequence $(\theta_n)_{n \in \Ni}$ in $(0,\frac{\pi}{2})$ 
by $\theta_1 = \frac{\pi}{2d}$
and $\theta_{n+1} = \theta_n + \frac{1}{d} \, (\frac{\pi}{2} - \theta_n)$
for all $n \in \Ni$.
We shall prove that for all $n \in \Ni$ and $z_0 \in \Sigma(\theta_n)$ with $|z_0| = 1$
the semigroup $(S^V_{t z_0})_{t > 0}$ satisfies Poisson bounds.
The proof is by induction.

The case $n = 1$ is proved in Lemma~\ref{ldtnpcont306}.
Let $n \in \Ni$ and suppose that the semigroup $(S^V_{t z_0})_{t > 0}$ satisfies Poisson bounds
for all $z_0 \in \Sigma(\theta_n)$ with $|z_0| = 1$.
Let $z_0 \in \Sigma(\theta_{n+1})$ with $|z_0| = 1$.
There exist $z_1 \in \Sigma(\theta_n)$ and $z_2 \in \Sigma(\frac{1}{d} \, (\frac{\pi}{2} - \theta_n))$
such that $|z_1| = |z_2| = 1$ and $z_0 = z_1 \, z_2$.
Then the $C_0$-semigroup $(S^V_{t z_1})_{t > 0}$ satisfies Poisson bounds by the induction 
hypothesis and it extends to a holomorphic semigroup on $L_2(\Gamma)$ 
with angle $\frac{\pi}{2} - \theta_n$.
Arguing as in the proof of Lemma~\ref{ldtnpcont306}, starting with the 
Poisson bounds for $(S^V_{t z_1})_{t > 0}$, it follows that 
the semigroup $(S^V_{t z_1 z_2})_{t > 0}$ satisfies Poisson bounds.
Then by induction the claim follows.

Clearly $(\theta_n)_{n \in \Ni}$ is an increasing sequence and its limit is $\frac{\pi}{2}$.
In particular there exists an $n \in \Ni$ such that $\theta < \theta_n$.
So we need to use the above induction step only a finite number of times to 
cover $\Sigma(\theta)$.
Therefore there exist $c,\delta > 0$ such that 
\[
| K^V_z(w_1,w_2) | 
\leq \frac{c \, (1 \wedge \RRe z)^{-(d-1)} \, e^{\delta \RRe z}}
         {\displaystyle \Big( 1 + \frac{|w_1-w_2|}{ |z|} \Big)^d }
\]
for all $z \in \Sigma(\theta)$ and $w_1,w_2 \in \Gamma$. 
If $t > \frac{1}{2}$, then
\[
\| S^V_t \|_{2 \to \infty} 
= \| S^V_{\frac{1}{2}} \, S^V_{t-\frac{1}{2}}  \|_{2 \to \infty} 
\le \| S^V_{\frac{1}{2}} \|_{2 \to \infty} \| S^V_{t -\frac{1}{2}} \|_{2 \to 2}.
\]
Hence there exists a $c > 0$ such that 
\[
\| S^V_t \|_{2 \to \infty} 
\le c e^{-\lambda_1 t}
\]
for all $t > 0$.
Consequently for all $z = t+ is \in \Sigma(\frac{\pi}{2})$ with  
$ t = \RRe z > 1$ we obtain from 
the identity $S^V_z = S^V_{\frac{t}{2}} \, S^V_{is} \, S^V_{\frac{t}{2}} $ that 
\[
\|  S^V_z \|_{1 \to \infty} 
\le \|  S^V_{\frac{t}{2}} \|_{2 \to \infty}  
       \|  S^V_{\frac{t}{2}} \|_{1 \to 2}  
       \|  S^V_{is} \|_{2 \to 2}  
\le c^2 \, e^{-\lambda_1 \RRe z}.
\]
This implies 
\[
 | K^V_z(w_1,w_2) | \le c^2 e^{-\lambda_1 \RRe z}
\]
for all $w_1,w_2 \in \Gamma$.
 The desired estimate for $\RRe z > 1$ follows from this latter estimate since
 \[
\Big( 1 + \frac{|w_1-w_2|}{ |z|} \Big)^d 
\]
 is bounded uniformly in $z$ and $(w_1, w_2)$, since the  set $\Gamma$ is bounded.
\end{proof}

\begin{prop} \label{pdtnpcont308} 
The semigroup $T^V$ is holomorphic on $C(\Gamma)$ with angle $\frac{\pi}{2}$.
For all $p \in [1,\infty)$ semigroup $S^V$ is holomorphic on $L_p(\Gamma)$ with angle $\frac{\pi}{2}$.
\end{prop}
\begin{proof}
Let $\theta \in (0,\frac{\pi}{2})$.
By Theorem~\ref{tdtnpcont307} there are $c,\delta > 0$ such that 
$\|S^V_z\|_{p \to p} \leq c \, e^{\delta \RRe z}$ for all $p \in [1,\infty]$ and 
$z \in \Sigma(\theta)$.
Given $z \in  \Sigma(\theta)$ we find a $t > 0$ small enough 
such that $z- t \in \Sigma(\theta)$.
By the semigroup property we can write $S^V_z = S^V_t \, S^V_{z-t}$.
Then Lemma~\ref{ldtnpcont202} gives $S^V_z (L_\infty(\Gamma)) \subset C(\Gamma)$.
The restriction $T^V_z$ of $S^V_z$ is a bounded operator on $C(\Gamma)$. 

In order to see that $z \mapsto T^V_z $ is holomorphic from $\Sigma(\theta)$ 
with values in $\cl(C(\Gamma))$ it is enough to prove that
$z \mapsto S^V_z $ is holomorphic from $\Sigma(\theta)$ 
with values in $\cl(L_\infty(\Gamma))$.
This follows from the holomorphy of the semigroup $S^V$ on $L_2(\Gamma)$ 
(recall that $\cn_V$ is a self-adjoint operator) and \cite{ABHN}, Proposition~A.3 
by choosing there 
\[
N = \{ f \in  L_1(\Gamma) \cap L_2(\Gamma): \| f \|_2 \le 1 \}
.  \]

Finally, the fact that $z \mapsto T^V_z$ is strongly continuous on $C(\Gamma)$ 
follows as at the end of the proof of Theorem~2.4 in \cite{Ouh4}.
The holomorphy of $S^V$ on $L_p(\Gamma)$ follows then by duality and 
interpolation, or, alternatively, by using \cite{Kat1} Theorem~IX.1.23.
\end{proof}

\begin{proof}[{\bf Proof of Theorem~\ref{tdtnpcont101}.}]
This is a combination of Propositions~\ref{pdtnpcont203} and~\ref{pdtnpcont308}.
\end{proof}

As a corollary of Theorem~\ref{tdtnpcont307} 
one obtains a holomorphic $H^\infty$-functional calculus on $L_p$
with optimal angle.

\begin{cor} \label{cdtnpcont309}
Suppose $\Omega \subset \Ri^d$ is bounded connected with a 
$C^{1+ \kappa}$-boundary $\Gamma$ for some $\kappa \in (0,1)$.
Suppose also that each $c_{kl} = c_{lk}$ is real valued and 
H\"older continuous on~$\Omega$.
Let $V \in L_\infty(\Omega,\Ri)$ and suppose that 
$0 \notin \sigma(A_D +V)$.
Let $\lambda > \min \sigma(A_D +V)$ and $\theta \in (0,\frac{\pi}{2})$.
Then the operator $\cn_V + \lambda \, I$ has a bounded 
$H^\infty(\Sigma(\theta))$-functional calculus on $L_p(\Gamma)$ for all 
$p \in (1,\infty)$.
\end{cor}
\begin{proof}
This follows from Theorem~\ref{tdtnpcont307} and \cite{DR1} Theorem~3.1.
\end{proof}

Consequently, the operator $\cn_V + \lambda \, I$ has bounded imaginary powers.

\begin{cor} \label{cdtnpcont310}
Adopt the notation and assumptions of Corollary~\ref{cdtnpcont309}.
Let $\nu > 0$ and $p \in (1,\infty)$.
Then there exists a $c > 0$ such that 
\[
\|(\cn_V + \lambda \, I)^{is}\|_{L_p(\Gamma) \to L_p(\Gamma)}
\leq c \, e^{\nu |s|}
\]
for all $s \in \Ri$.
\end{cor}

An interesting application of the imaginary powers of the 
previous corollary is   maximal regularity for the parabolic 
problem
\begin{equation}
\left\{ \begin{array}{l}
\displaystyle \frac{\textstyle \partial \varphi(t,\cdot)}{\textstyle \partial t} 
     + \cn_V \varphi(t,\cdot) = f(t)  
\qquad (t \in (0,\tau])   \\[10pt]
\varphi(0) = \varphi_0
        \end{array} \right.
\label{ecdtnpcont311;1}
\end{equation}
for any $\tau > 0$ and $\varphi_0 \in (L_p(\Gamma), D(\cn_V))_{1-\frac{1}{p},p}$,
where $(L_p(\Gamma), D(\cn_V))_{1-\frac{1}{p},p}$ is the real interpolation space.

\begin{cor} \label{cdtnpcont311}
Adopt the notation and assumptions of Corollary~\ref{cdtnpcont309}.
Then the operator $\cn_V$ has maximal $L_r$-regularity on $L_p(\Gamma)$
for all $p,r \in (1,\infty)$.
More precisely, for every $p,r \in (1,\infty)$ and $\tau > 0$ there exists a $c > 0$ 
such that for all $f \in L_r(0,\tau,L_p(\Gamma))$ and 
$\varphi_0 \in (L_p(\Gamma), D(\cn_V))_{1-\frac{1}{p},p}$ there exists a unique 
solution $\varphi$ to the problem {\rm (\ref{ecdtnpcont311;1})}
with $\varphi \in W^{1,r}(0,\tau,L_p(\Gamma)) \cap L_r(0,\tau,D(\cn_V))$
and 
\[
\|\varphi\|_{W^{1,r}(0,\tau,L_p(\Gamma))} + \|\varphi\|_{L_r(0,\tau,D(\cn_V))}
\leq c \Big( \|\varphi_0\|_{(L_p(\Gamma), D(\cn_V))_{1-\frac{1}{p},p}} 
             + \|f\|_{L_r(0,\tau,L_p(\Gamma))}
       \Big)
.  \]
\end{cor}
\begin{proof}
This is an application of the Dore--Venni theorem \cite{DV} Theorem~3.3
and \cite{Lun} Proposition~1.2.10.
\end{proof}

\subsection*{Acknowledgements} 
This work was carried out when the second named author was visiting the University of
Auckland and the first named author was visiting the University of Bordeaux.
Both authors wish to thank the universities for hospitalities.
The research of A.F.M. ter  Elst  is partly supported by the 
Marsden Fund Council from Government funding, 
administered by the Royal Society of New Zealand. 
The research of E.M.  Ouhabaz  is partly supported by the ANR 
project `Harmonic Analysis at its Boundaries',  ANR-12-BS01-0013-02.

\end{document}